\documentclass[11pt, a4paper]{article}
\usepackage{amsfonts,amssymb,amsmath,amscd,latexsym,theorem}
\usepackage{color}
\title{Fractional Adams-Moser-Trudinger type inequalities}

\author{Luca Martinazzi\thanks{supported by Swiss National Science Foundation, project nr. PP00P2-144669.}\\ \small{University of Basel}\\
\small \texttt{luca.martinazzi@unibas.ch}}

\newtheorem{trm}{Theorem}
\newtheorem{prop}[trm]{Proposition}

\newtheorem{lemma}[trm]{Lemma}

\newtheorem{rmk}[trm]{Remark}

\newcommand{\R}{\mathbb{R}}
\newcommand{\de}{\partial}
\newcommand{\ve}{\varepsilon}
\newcommand{\M}[1]{\mathcal{#1}}

\newenvironment{proof}{\noindent\emph{Proof.}}{\phantom{ } \hfill$\square$\medskip}

\newcommand{\vrac}[1]{\| #1 \|}

\newcommand{\eps}{\varepsilon}

\DeclareMathOperator{\loc}{loc}

\DeclareMathOperator*{\dist}{dist}

\begin{document}

\maketitle

\begin{abstract} Extending several works, we prove a general Adams-Moser-Trudinger type inequality for the embedding of Bessel-potential spaces $\tilde H^{\frac{n}{p},p}(\Omega)$ into Orlicz spaces for an arbitrary domain $\Omega$ with finite measure. In particular we prove
$$ \sup_{u\in \tilde H^{\frac{n}{p},p}(\Omega), \;\|(-\Delta)^{\frac{n}{2p}}u\|_{L^{p}(\Omega)}\leq 1}\int_{\Omega}e^{\alpha_{n,p} |u|^\frac{p}{p-1}}dx \leq c_{n,p}|\Omega|,
$$
for a positive constant $\alpha_{n,p}$ whose sharpness we also prove. We further extend this result to the case of Lorentz-spaces (i.e. $(-\Delta)^\frac{n}{2p}u\in L^{(p,q)})$. The proofs are simple, as they use Green functions for fractional Laplace operators and suitable cut-off procedures to reduce the fractional results to the sharp estimate on the Riesz potential proven by Adams and its generalization proven by Xiao and Zhai.

We also discuss an application to the problem of prescribing the $Q$-curvature and some open problems.
\end{abstract}

\section{Introduction}

Let $\Omega\subset\R^n$ be an open domain with finite measure $|\Omega|$. It is well known that for a positive integer $k<n$ and for $1\le p<\frac{n}{k}$ the Sobolev space $W^{k,p}_0(\Omega)$ embeds continuously into $L^{\frac{np}{n-kp}}(\Omega)$, while in the borderline case $p=\frac{n}{k}$ one has $W^{k,\frac{n}{k}}_0(\Omega)\not\subset L^\infty(\Omega)$, unless $k=n$. On the other hand, as shown by Yudovich \cite{yud}, Pohozaev \cite{poh}, Trudinger \cite{tru} and others, for the case $k=1$ one has
$$W^{1,n}_0(\Omega)\subset \left\{u\in L^1(\Omega):E_\beta(u):=\int_\Omega e^{\beta |u|^{\frac{n}{n-1}}}dx<\infty  \right\},\quad \text{for any }\beta <\infty,$$
and the functional $E_\beta$ is continuous on $W^{1,n}_0(\Omega)$.
This embedding was complemented with a sharp inequality by Moser \cite{mos}, the so-called Moser-Trudinger inequality: 
\begin{equation}\label{stimaMT0}
\sup_{u\in W^{1,n}_0(\Omega),\;\|\nabla u\|_{L^n(\Omega)}\le 1} \int_\Omega e^{\alpha_n |u|^\frac{n}{n-1}}dx \le C|\Omega|,\quad \alpha_n:=n\omega_{n-1}^\frac{1}{n-1},
\end{equation}
where $\omega_{n-1}$ is the volume of the unit sphere in $\R^n$. The constant $\alpha_n$ is sharp in the sense that for $\alpha>\alpha_n$ the supremum in \eqref{stimaMT0} is infinite.


An extension of Moser's result to the case $k>1$ was given by Adams \cite{ada}  who proved that
\begin{equation}\label{stimaMT0b}
\sup_{u\in C^k(\R^{n}),\;\mathrm{supp}(u)\subset\bar\Omega,\;\|\Delta^\frac{k}{2} u\|_{L^\frac{n}{k}(\Omega)}\le 1} \int_\Omega e^{\alpha |u|^\frac{n}{n-k}}dx \le C|\Omega|,
\end{equation}
for an optimal constant $\alpha=\alpha(k,n)$. Here $k\in (0,n)\cap\mathbb{N}$ and $\Delta^\frac{k}{2} u:=\nabla\Delta^\frac{k-1}{2}u$ when $k$ is odd. 

In this paper we study the fractional case of Adams' inequality, i.e. we allow $k\in (0,n)$ to be non-integer. 
Let us consider the space
\begin{equation*}
L_{s}(\R^n)=\left\{ u\in L^{1}_{\loc}(\R^n): \int_{\R^n}\frac{|u(x)|}{1+|x|^{n+s}}dx <\infty \right\}.
\end{equation*}
For functions $u\in L_{s}(\R^n)$ the fractional Laplacian  $(-\Delta)^\frac{s}{2}u$ can be defined as follows. First set 
$$(-\Delta)^\frac{s}{2}\varphi:= \mathcal{F}^{-1}(|\xi|^{s}\mathcal{F}{\varphi}(\xi))$$
for $\varphi$ belonging to the Schwartz space $\mathcal{S}(\R^n)$ of rapidly decreasing functions, where $\mathcal{F}$ denotes the unitary Fourier transform. Then for $u\in L_{s}(\mathbb{R})$ we define $(-\Delta)^\frac{s}{2}u$ as a tempered distributions via the formula
\begin{equation}\label{deffraclap}
\langle (-\Delta)^\frac{s}{2}u,\varphi \rangle = \langle u,(-\Delta)^\frac{s}{2}\varphi \rangle:=\int_{\mathbb{R}} u (-\Delta)^\frac{s}{2}\varphi dx, \quad \varphi\in\mathcal{S}(\R^n),
\end{equation} 
the right-hand side being well-defined because
$$|(-\Delta)^\frac{s}{2}\varphi(x)|\le \frac{C_\varphi}{1+|x|^{n+s}},\quad \text{for every }\varphi \in \mathcal{S}(\R^n).$$
see e.g. Proposition 2.1 in \cite{hyd2}.

For a set  $\Omega\subset  \R^{n}$ (possibly unbounded), $s \ge 0$ and $p\in (1,\infty)$ we define 
\begin{equation}\label{defHsp}
\begin{split}
H^{s,p}(\R^n)&:=\{u\in L^p(\R^n): (-\Delta)^{\frac s2} u\in L^p(\R^n)\}\\
\tilde H^{s,p}(\Omega)&:=\{u\in H^{s,p}(\R^n): u\equiv 0\text{ in }\R^n\setminus \Omega\}.
\end{split}
\end{equation}
Then we have:

\begin{trm}\label{MT2} For any $p\in (1,\infty)$ and positive integer $n$ set
\begin{equation}\label{defalpha}
K_{n,s}:=\frac{\Gamma((n-s)/2)}{\Gamma(s/2) 2^{s}\pi^{n/2}},\quad \alpha_{n,p}:=\frac{n}{\omega_{n-1}} K_{n,\frac{n}{p}}^{-p'},\quad p':=\frac{p}{p-1}.
\end{equation}
Then for any open set $\Omega\subset \R^{n}$ with finite measure we have
\begin{equation}\label{stimaMT2}
\sup_{u\in \tilde H^{\frac{n}{p},p}(\Omega), \;\|(-\Delta)^{\frac{n}{2p}}u\|_{L^{p}(\Omega)}\leq 1}\int_{\Omega} e^{\alpha_{n,p} |u|^{p'}}dx \leq c_{n,p}|\Omega|.
\end{equation}
Moreover the constant $\alpha_{n,p}$ is sharp in the sense that
we cannot replace it with any larger one without making the supremum in \eqref{stimaMT2} infinite.
\end{trm}

\begin{rmk} The norm $\|(-\Delta)^{\frac{n}{2p}}u\|_{L^{p}(\Omega)}$ is equivalent to $\|(-\Delta)^{\frac{n}{2p}}u\|_{L^{p}(\R^{n})}$ for functions in $\tilde H^{\frac{n}{p},p}(\Omega)$,
see for instance Theorem 7.1 in \cite{Gr0}.
\end{rmk}

To explain the idea of the proof let us recall that Adams' result \eqref{stimaMT0b} follows at once from the following result, which is Theorem 2 in \cite{ada}:

\begin{trm}[Adams]\label{trmadams}
Let $\Omega \subset\R^{n}$ be an open set with finite measure $|\Omega|$, and fix $p\in (1,\infty)$. For $\alpha\in (0,n)$ and $f\in L^p(\Omega)$ consider the Riesz potential $I_{\alpha}f$ defined as
$$I_\alpha f(x)=\int_\Omega \frac{f(y)}{|x-y|^{n-\alpha}}dy.$$
Then
$$\sup_{f\in L^p(\Omega),\;\|f\|_{L^p(\Omega)}\le 1} \int_{\Omega}e^{\frac n{\omega_{n-1}} |I_{\frac np}f|^{p'}}dx\le c_{n,p}|\Omega|,\quad p'=\frac{p}{p-1}.$$
The constant $\frac{n}{\omega_{n-1}}$ is sharp in the sense that
$$\sup_{f\in C^\infty_0(B_\delta),\;\|f\|_{L^p(B_\delta)}\le 1} \int_{B_\delta}e^{\gamma |I_{\frac np}f|^{p'}}dx=\infty \quad \text{for every }\delta>0,\;\gamma>\frac{n}{\omega_{n-1}}. $$
\end{trm}

Adams applies this result to the function $f= (-\Delta)^\frac{n}{2p}u$ where $u$ is smooth and supported in $\bar \Omega$, and $p=\frac{n}{k}$ (compare to \eqref{stimaMT0b}). Here it is crucial that when $\frac{n}{p}\in \mathbb{N}$, then the support of $f$ (with Adams' convention that $(-\Delta)^{\frac {n}{2p}}=\nabla(-\Delta)^\frac{n/p-1}{2}$ for $\frac{n}{p}$ odd, up to a sign) does not exceed the support of $u$, so that Theorem \ref{trmadams} can be applied. This is not the case when $\frac{n}{p}\not \in \mathbb{N}$. Indeed for general $s>0$ the support of $(-\Delta)^\frac{s}{2} u$ can be the whole $\R^n$ even if $u$ is compactly supported.

In order to circumvent this issue, instead of using the Riesz potential we will write $u$ in terms of a Green representation formula (Proposition \ref{propgreen2} below)
\begin{equation}\label{greenform}
u(x)=\int_{\Omega} G_\frac{n}{p}(x,y)(-\Delta)^\frac{n}{2p} u(y)dy,
\end{equation}
which holds for a suitable Green function which we construct using variational methods, and which we can sharply bound in terms of the fundamental solution of $(-\Delta)^\frac{n}{2p}$ in $\R^n$ (see estimate \eqref{G<Fbis} in particular). The Green formula \eqref{greenform} will be first proven for functions in $C^\infty_c(\Omega)$, and then extended to all functions in $\tilde H^{\frac np,p}(\Omega)$ thanks to a density theorem of Yu. V. Netrusov. Since $\Omega$ is not necessarily bounded and might have rough boundary, we must be careful, particularly in using maximum principles (we will use a simple ``variational'' maximum principle instead of the one of Silvestre \cite{Sil}). We remark that estimates for the Green function of $(-\Delta)^\frac{s}{2}$ on bounded domains with $C^{1,1}$ boundary were proven by Chen and Song \cite{CS} and other authors (see e.g \cite{Aba}) when $s<2$. This is of course insufficient for our purposes. 
Our strategy here is to first prove the precise estimate for $G_\sigma$ when $\sigma\in (0,2]$ (only assuming $|\Omega|<\infty$), and then, following a suggestion of A. Maalaoui, write $G_s$ as convolution of $k$ copies of $G_2$ and one copy of $G_\sigma$ for $s=2k+\sigma$.

The sharpness of the inequality \eqref{stimaMT2}, i.e. of the constant $\alpha_{n,p}$ will be instead obtained by constructing suitable test functions, with a method of cut-off suggested by A. Schikorra, and using a disjoint-support estimate (Proposition \ref{propdisj} below) which extends analogous estimates from \cite{MMS}.

\medskip

Let us mention some previous partial results. Extending an early result of Strichartz \cite{str}, Ozawa \cite{oza} proved a subcritical version of Theorem \ref{MT2}, i.e. \eqref{stimaMT2} for some $\alpha<\alpha_{n,p}$ under some regularity assumptions on $\Omega$ (for instance $\Omega$ bounded and with regular boundary, or with the extension property). Lam and Lu \cite{LaL} proved that for $\Omega=\R^n$ the integral in \eqref{stimaMT2} is uniformly bounded for $u$ such that $\|(\tau I-\Delta)^\frac{n}{2p}u\|_{L^p(\R^n)}\le 1$ (here $\tau>0$ is fixed). More recently Iula, Maalaoui and Martinazzi \cite{IMM} proved Theorem \ref{MT2} in dimension $1$, i.e. on a bounded interval $I\Subset \R$ and for the sharp constant $\alpha_{1,p}$. They also proved the following sharpness result:
\begin{equation}\label{stimaMT2c}
\sup_{u\in \tilde H^{\frac{1}{p},p}(I), \;\|(-\Delta)^{\frac{1}{2p}}u\|_{L^{p}(I)}\leq 1}\int_{I}|u|^a e^{\alpha_{1,p} |u|^\frac{p}{p-1}}dx =\infty\quad \text{for every }a>0.
\end{equation}
Notice that this is stronger than just saying that $\alpha_{1,p}$ is optimal. In Theorem \ref{MT2} we are not able to prove the analog of \eqref{stimaMT2c} because we use cut-off functions, which are convenient, but difficult to estimate when handling fractional norms. The proof of \eqref{stimaMT2c} instead relies on constructing test functions of the form
$$u(x)=\int_I G_{\frac 1p}(x,y) f(y)dy,$$
where $f$ is suitably prescribed, and on the fact that such $u$ belongs to $\tilde H^{\frac 1p,p}(I)$. The same does not hold when dealing with operators of order $\frac np>1$, since $u$ defined via a Green representation formula might not be regular enough at the boundary to belong to the right space $\tilde H^{\frac{n}{p},p}(\Omega)$.

\medskip

Recently, extending results of Cassani and Tarsi \cite{CT}, Xiao and Zhai \cite{XZ} considered a fractional Adams' type inequality under the assumption that $(-\Delta)^\frac{n}{2p} u$ is supported in $\Omega$ (which is not implied by and in general not compatible with our request that $u$ itself is supported in $\Omega$). In their work they extend the above-mentioned Adams' Theorem \ref{trmadams} to several situations, in particular considering $f$ belonging to the Lorentz space $L^{(p,q)}(\Omega)$ (when $\frac{n}{p}\in \mathbb{N}$ this had been previously done by Alberico \cite{Alb}). For further extensions we refer to the work of Fontana and Morpurgo \cite{FM}.

\begin{trm}[Xiao-Zhai]\label{XZ} Let $\Omega\subset\R^n$ be open and have finite measure, let $p\in (1,\infty)$, and let the Riesz potential $I_\alpha$ be defined as in Theorem \ref{trmadams}. For $q\in (1,\infty]$ set $\gamma_{n,p,q}:=\big(\frac{n}{\omega_{n-1}}\big)^\frac{q'}{p'}$. Then for $q\in (1,\infty)$ one has
$$\sup_{f\in L^{(p,q)}(\Omega), \;\|f\|_{L^{(p,q)}(\Omega)}\le 1}\int_\Omega e^{\gamma_{n,p,q}|I_\frac{n}{p} f|^{q'}}dx\le c_{n,p,q}|\Omega|,$$
and the constant $\gamma_{n,p,q}$ is sharp.
When $q=1$
$$\|I_\frac np f\|_{L^\infty(\Omega)}\le \left(\frac{\omega_{n-1}}{n}\right)^\frac{1}{p'}\|f\|_{L^{(p,1)}(\Omega)},\quad \text{for every }f\in L^{(p,1)}(\Omega).$$
Finally when $q=\infty$ (and by convention $q'=1$)
$$\sup_{f\in L^{(p,\infty)}(\Omega),\;\|f\|_{L^{(p,\infty)}(\Omega)}\le 1}\int_\Omega e^{\gamma|I_\frac{n}{p} f|}dx \le \frac{d_{n,p}|\Omega|}{\gamma_{n,p,\infty} -\gamma}\quad \text{for every }\gamma<\gamma_{n,p,\infty},$$
and the constant $\gamma_{n,p,\infty}$ cannot be replaced by a larger one.
\end{trm}

 Still resting on the Green representation formula \eqref{greenform}, Xiao and Zhai's results can be immediately extended to the case of functions supported in $\Omega$ without any assumption on the support of their fractional derivatives. More precisely for $q\in [1,\infty]$ let $\tilde H^{s,(p,q)}(\Omega)$ denote the closure of $C^\infty_c(\Omega)$ under the norm
\begin{equation}\label{defHspq}
\|u\|_{H^{s,(p,q)}(\R^n)}:=\|u\|_{L^{(p,q)}(\R^n)} +\|(-\Delta)^\frac{s}{2} u\|_{L^{(p,q)}(\R^n)}.
\end{equation}
Notice that by Netrusov's theorem (Theorem \ref{lemmadens} in the appendix) $\tilde H^{\frac{n}{p},(p,p)}(\Omega)=\tilde H^{\frac{n}{p},p}(\Omega)$, the latter space being defined in \eqref{defHsp}. We then obtain:

\begin{trm}\label{MT3} Let $\Omega$, $p$, $\alpha_{n,p}$ and $K_{n,p}$ be as in Theorem \ref{MT2}. For $q\in (1,\infty]$ set $\beta_{n,p,q}:=(\alpha_{n,p})^\frac{q'}{p'}$.
Then for $q\in (1,\infty)$
\begin{equation}\label{stimaMTL}
\sup_{u\in \tilde H^{\frac{n}{p},(p,q)}(\Omega), \;\|(-\Delta)^{\frac{n}{2p}}u\|_{L^{(p,q)}(\Omega)}\leq 1}\int_{\Omega} e^{\beta_{n,p,q}|u|^{q'}}dx \leq c_{n,p,q}|\Omega|,
\end{equation}
and the constant $\beta_{n,p,q}$ is sharp.
When $q=1$
\begin{equation}\label{stimaMTL2}
\|u\|_{L^\infty(\Omega)}\le (\alpha_{n,p})^{-\frac{1}{p'}} \|(-\Delta)^\frac{n}{2p} u\|_{L^{(p,1)}(\Omega)},\quad \text{for every } u\in \tilde H^{\frac{n}{p},(p,1)}(\Omega).
\end{equation}
Finally when $q=\infty$ we get
\begin{equation}\label{stimaMTL3}
\sup_{u\in \tilde H^{\frac{n}{p},(p,\infty)}(\Omega), \;\|(-\Delta)^{\frac{n}{2p}}u\|_{L^{(p,\infty)}(\Omega)}\leq 1}\int_{\Omega} e^{\beta|u|}dx \leq \frac{K_{n,\frac{n}{p}}^{-1} d_{n,p}|\Omega|}{\beta_{n,p,\infty}-\beta},\quad \text{for every }\beta<\beta_{n,p,\infty}.
\end{equation}
The constant $\beta_{n,p,\infty}$ in \eqref{stimaMTL3} cannot be replaced by a larger one, . The constants $c_{n,p,q}$ and $d_{n,p}$ are as in Theorem \ref{XZ}.
\end{trm}

We mention that Adams-Moser-Trudinger type inequalities of integer order on manifolds have been proven by Fontana \cite{fon}. In the case $p=1$ related inequalities (similar to \eqref{stimaMTL3}) have been originally proven by Br\'ezis and Merle \cite{BM} in dimension $2$, and then extended by C-S. Lin \cite{lin}, J-C. Wei \cite{wei} and the author \cite{mar1} to arbitrary even dimension, and recently by Da Lio-Martinazzi-Rivi\`ere \cite{DLMR} in dimension $n=1$ and by A. Hyder \cite{hyd2} in arbitrary odd dimension.

\medskip

The paper is organized as follows. In Section \ref{sec:2} we will prove Theorems \ref{MT2} and \ref{MT3}. In section \ref{sec:3} we will discuss a couple of applications to semilinear equations involving exponential nonlinearities, including those arising in the prescribed $Q$-curvature problem. Open questions are discussed in Section \ref{sec:4}, while in the Appendix we collect some known results which we need for the proofs in Section \ref{sec:2}.

\section{Green functions and proof of Theorems \ref{MT2} and \ref{MT3}}\label{sec:2}

The following lemma is well known. One can prove it by hands using \eqref{deffraclap} and the formula for the Fourier transform of $|x|^{s-n}$, see e.g. \cite[Theorem 5.9]{LL}.

\begin{lemma}\label{green2}
The fundamental solution of $(-\Delta)^\frac{s}{2}$ on $\R^{n}$ is
$F_s(x)=K_{n,s}|x|^{s-n},$
in the sense that $F_s\in L_s (\R^n)$ and
$(-\Delta)^\frac{s}{2}F_s=\delta_0$
in the sense of tempered distributions (see \eqref{deffraclap}). Moreover
$$(-\Delta)^\frac{s}{2}(F_s* f)=f\quad \text{for every }f\in \mathcal{S}(\R^n).$$
\end{lemma}

\begin{prop}\label{propgreen}
Let $\Omega\subset\R^n$ be open and have finite measure and $\sigma\in (0,2]$ such that $\sigma<n$ be fixed. Then for every $x\in \Omega$ there is a function $G_\sigma(x,\cdot)\in L^1(\R^n)$ satisfying
\begin{equation}\label{eqgreen}
\left\{
\begin{array}{ll}
(-\Delta)^\frac{\sigma}{2} G_\sigma(x,\cdot)=\delta_x&\text{in } \Omega\\
G_\sigma(x,\cdot)\equiv 0&\text{in } \R^n\setminus \Omega,
\end{array}
\right.
\end{equation}
the first equation being in the sense of distributions (i.e. as in Proposition \ref{trmexist} in the appendix). Moreover
\begin{equation}\label{G<F}
0\le G_\sigma(x,y)\leq F_\sigma(x-y) \quad \text{for a.e. }y\ne x\in \Omega.
\end{equation}
Finally given $u\in \tilde H^{\sigma,p}(\Omega)$ for some $p\ge 1$, we have
\begin{equation}\label{eqgreenb}
u(x) = \int_\Omega G_\sigma(x,y)(-\Delta)^\frac{\sigma}{2} u(y)dy,\quad  \text{for a.e. }x\in \Omega,
\end{equation}
where the right-hand side is well defined for a.e. $x\in \Omega$ thanks to \eqref{G<F} and Fubini's theorem.
\end{prop}

\begin{proof} We first consider that case $\sigma<2$. Given $x\in \Omega$ let $\delta=\frac{1}{2}\dist(x,\de\Omega)$ and $g_x\in C^1(\R^{n})$ be any function with $g_x(y)=F_\sigma(x-y)$ for $y\in\R^{n}\setminus B_\delta(x)$. We first claim that $g_x$ satisfies
\begin{equation}\label{stimagx}
\int_\Omega\int_{\R^n}\frac{(g_x(y)-g_x(z))^2}{|y-z|^{n+\sigma}}dydz<\infty.
\end{equation}
Indeed, splitting for a fixed $z\in \Omega$
\[
\begin{split}
\int_{\R^n}\frac{(g_x(y)-g_x(z))^2}{|y-z|^{n+\sigma}}dy&= \int_{B_1(z)}\frac{(g_x(y)-g_x(z))^2}{|y-z|^{n+\sigma}}dy +\int_{\R^n\setminus B_1(z)}\frac{(g_x(y)-g_x(z))^2}{|y-z|^{n+\sigma}}dy\\
&=:(I)+(II)
\end{split}
\]
we easily see that with a constant $C$ only depending on $n$ and $\sigma$
$$(I)\le \|\nabla g_x\|_{L^\infty(B_1(z))}^2 \int_{B_1(z)}\frac{1}{|y-z|^{n+\sigma-2}}dy\le C\|\nabla g_x\|_{L^\infty(\R^n)}^2,$$
(where we used that $|g_x(y)-g_x(z)|\le \|\nabla g_x\|_{L^\infty(B_1(z))}|y-z|$) and
$$(II)\le \|g_x\|_{L^\infty(\R^n)}^2\int_{\R^n\setminus B_1(z)}\frac{1}{|y-z|^{n+\sigma}}dy \le C \|g_x\|_{L^\infty(\R^n)}^2,$$
and integrating with respect to $z$ on $\Omega$ (which has finite measure) we infer that \eqref{stimagx} holds, as claimed.

\medskip

Now Proposition \ref{trmexist} in the appendix implies that there exists a unique $H_\sigma(x,\cdot)\in \tilde H^{\frac \sigma 2,2}(\Omega)+g_x$ solution to 
\begin{equation}\label{H}
\left\{
\begin{array}{ll}
(-\Delta)^\frac{\sigma}{2} H_\sigma(x,\cdot)=0 &\text{in } \Omega\\
H_\sigma(x,\cdot )=g_x &\text{in } \R^{n}\setminus \Omega,
\end{array}
\right.
\end{equation}
in the sense of distribution. Moreover, by Proposition \ref{maxprinc2} in the appendix applied to the functions $u_1= H_\sigma(x,\cdot)$ and $u_2=-H_\sigma(x,\cdot)+\sup_{\R^n\setminus\Omega} g_x$, we infer
\begin{equation}\label{stimaHmax}
0\le H_\sigma(x,y)\le  \sup_{z\in\R^n\setminus \Omega} F_\sigma(x-z)\quad \text{for a.e. }y\in \Omega.
\end{equation} 
Notice that here we used that
$$u_1\in \tilde H^{\frac \sigma2,2}(\Omega)+g_x,\quad u_2\in \tilde H^{\frac \sigma2,2}(\Omega)-g_x+\sup_{\R^n\setminus\Omega} g_x,$$
and the functions $g_1=g_x$ and $g_2=-g_x+\sup_{\R^n\setminus\Omega} g_x$ satisfy \eqref{intg} thanks to \eqref{stimagx}.

Set
$$G_\sigma(x,y):=F_\sigma(x-y)-H_\sigma(x,y),\quad (x,y)\in \Omega\times \R^n.$$
That $G_\sigma(x,\cdot)$ satisfies \eqref{eqgreen} follows at once from Lemma \ref{green2} and \eqref{H}. We also have $G_\sigma(x,y)\le F_\sigma(x-y)$ thanks to \eqref{stimaHmax}. 


We want to show that $G_\sigma(x,\cdot )\ge 0$ in $\Omega$. Since $H_\sigma(x,\cdot)$ is bounded, for $\ve\in (0,\delta]$ sufficiently small we have
$$F_\sigma(x-y)>H_\sigma(x,y) \quad \text{for a.e. }y\in B_\ve(x),$$
hence $G_\sigma(x,\cdot)\ge 0$ in $B_\ve(x)$.
We can now modify $F_\sigma(x-\cdot)$ in $B_\ve(x)$ to obtain a new function $\Gamma_x\in C^{1,1}(\R^n)$ with
$$\Gamma_x\le F_\sigma(x-\cdot)\text{ in }\R^n,\quad \Gamma_x=F_\sigma(x-\cdot)\text{ in }\R^n\setminus B_\ve(x),\quad (-\Delta)^\frac{\sigma}{2} \Gamma_x\ge 0\text{ in }\R^n,$$ 
as done in \cite[Section 2.2]{Sil}, see \cite[Prop. 2.11]{Sil} in particular.  We now claim that
$$\Gamma_x- H_\sigma(x,\cdot)\in \tilde H^{\frac\sigma2,2}(\Omega).$$
Indeed the function $\Gamma_x-g_x$ lies in $C^1(\R^n)$ and vanishes outside $\Omega$. Then with the same computations used to prove \eqref{stimagx} one easily sees that $\Gamma_x-g_x\in W^{\frac\sigma2,2}(\R^n)=H^{\frac\sigma2,2}(\R^n)$ (see also Proposition \ref{HW}), so that 
$$\Gamma_x- H_\sigma(x,\cdot)= (\Gamma_x-g_x)- (H_\sigma(x,\cdot)-g_x) \in \tilde H^{\frac\sigma2,2}(\Omega),$$
as claimed. Then, since
$$ (-\Delta)^\frac{\sigma}{2} (\Gamma_x-H_\sigma(x,\cdot))\ge 0\text{ in }\Omega,\quad \Gamma_x-H_\sigma(x,\cdot)\equiv 0\text{ in }\R^n\setminus\Omega,$$
by the maximum principle (Proposition \ref{maxprinc2}) we have $\Gamma_x- H_\sigma(x,\cdot)\ge 0$ in $\Omega$, hence $G_\sigma(x,\cdot)\ge 0$ also in $\Omega\setminus B_\ve(x)$. This completes the proof of \eqref{G<F}.

To prove \eqref{eqgreenb}, let us start considering $u\in C^\infty_c(\Omega).$ Let $\delta_x$ denote the Dirac distribution in $x$. Then, using $u$ as test function in \eqref{eqgreen}, we get
$$u(x)=\langle \delta_x,u\rangle=\langle (-\Delta)^\frac{\sigma}{2} G_\sigma(x,\cdot),u\rangle:=\int_\Omega G_\sigma(x,y)(-\Delta)^\frac{\sigma}{2} u(y)dy.$$
Given now $u\in \tilde H^{\sigma,p}(\Omega)$, let $(u_k)_{k\in \mathbb{N}}\subset C^\infty_c(\Omega)$ converge to $u$ in $\tilde H^{\sigma,p}(\Omega)$, i.e.
$$u_k\to u,\quad (-\Delta)^\frac{\sigma}{2} u_k\to (-\Delta)^\frac{\sigma}{2} u\quad \text{in }L^p(\R^n), \text{ hence in }L^1(\Omega),$$
see Theorem \ref{lemmadens}. Then
$$u \overset{L^1(\Omega)}\longleftarrow u_k=\int_\Omega G_\sigma(\cdot ,y)(-\Delta)^\frac{\sigma}{2} u_k(y)dy\overset{L^1(\Omega)}\longrightarrow \int_\Omega G_\sigma(\cdot ,y)(-\Delta)^\frac{\sigma}{2} u(y)dy,$$
the convergence on the right following from \eqref{G<F} and Fubini's theorem:
\[
\begin{split}
\int_\Omega &\left|\int_\Omega G_\sigma(x,y)\left[(-\Delta)^\frac{\sigma}{2} u_k(y)- (-\Delta)^\frac{\sigma}{2} u(y)\right]dy \right|dx\\
&\le \int_\Omega \int_\Omega F_\sigma (x-y)\left|(-\Delta)^\frac{\sigma}{2} u_k(y)- (-\Delta)^\frac{\sigma}{2} u(y)\right| dx dy\\
&\le \sup_{y\in \Omega}\|F_\sigma(\cdot -y)\|_{L^1(\Omega)} \|(-\Delta)^\frac{\sigma}{2} u_k- (-\Delta)^\frac{\sigma}{2} u\|_{L^1(\Omega)}\to 0
\end{split}
\]
as $k\to \infty$, where we used that $\|F_\sigma(\cdot-y)\|_{L^1(\Omega)}\le C$ for a constant $C$ independent of $y$, as can be seen by writing
\[\begin{split}
\|F_\sigma(\cdot-y)\|_{L^1(\Omega)}&\le \|F_\sigma(\cdot-y)\|_{L^1(B_1(y))}+\|F_\sigma(\cdot-y)\|_{L^1(\Omega\setminus B_1(y))}\\
&\le C_{n,\sigma}+K_{n,\sigma}|\Omega|.
\end{split}\]
Since the convergence in $L^1$ implies the a.e. convergence (up to a subsequence), \eqref{eqgreenb} follows.

The case $\sigma=2$ is probably well known. The reader can easily prove it in a way similar to the case $\sigma\in (0,2)$, replacing Propositions \ref{trmexist} and \ref{maxprinc2} with the natural (local) analogs for $\sigma=2$. For instance the functional $\mathcal{B}_\sigma$ will be replaced by
$$\mathcal{B}_2(u,v):=\int_{\R^n}\nabla u\cdot \nabla v \,dx,\quad u,v\in H^{1,2}(\R^n).$$
To avoid confusion it might also be useful to notice that
\[\begin{split}
H^{1,2}_0(\Omega)&:=\{u\in L^2(\R^n):u\equiv 0\text{ in }\R^n\setminus \Omega, \nabla u\in L^2(\R^n)\}\\
&=\{u\in L^2(\R^n):u\equiv 0\text{ in }\R^n\setminus \Omega, (-\Delta)^\frac12 u\in L^2(\R^n)\}\\
&=:\tilde H^{1,2}(\Omega),
\end{split}\]
as can be seen via Fourier transform.
\end{proof}

Using the convolutions of several Green functions (an idea suggested by Ali Maalaoui) we can extend Proposition \ref{propgreen} to higher order $s>2$.

\begin{prop}\label{propgreen2} Let $\Omega\subset\R^n$ be open and have finite measure. Set $s=2k+\sigma<n$ with $k\in\mathbb{N}$, $\sigma\in (0,2]$, and define
$$G_s(x,y):=\int_\Omega G_2(x,y_1)\int_\Omega G_2(y_1,y_2)\dots\int_\Omega G_2(y_{k-1},y_k)\int_\Omega G_\sigma(y_k,y) dy_k\dots dy_1.$$
Then
\begin{equation}\label{G<Fbis}
0\le G_s(x,y)\le (\underbrace{F_2 * F_2*\dots *F_2}_{k\text{ times}}*F_\sigma) (x-y)= \frac{K_{n,s}}{|x-y|^{n-s}},
\end{equation}
where $*$ denotes the usual convolution in $\R^n$.
Moreover, if $u\in \tilde H^{s,p}(\Omega)$ for some $p\ge1$, it holds
\begin{equation}\label{greenrapr}
u(x)=\int_\Omega G_s(x,y)(-\Delta)^\frac{s}{2} u(y)dy, \quad \text{for a.e. }x\in\Omega.
\end{equation}
More generally \eqref{greenrapr} holds for functions $u:\R^n\to\R$ which can be approximated by a sequence $u_k\in C^\infty_c(\Omega)$ in the sense that $u_k\to u$ and $(-\Delta)^\frac{s}{2}u_k \to (-\Delta)^\frac{s}{2}u$ in $L^1(\Omega)$.
\end{prop}

\begin{proof}
Using \eqref{G<F} we immediately infer \eqref{G<Fbis}, where the right-hand side can be computed explicitly using Lemma \ref{green2} 
and the formula
$$|\cdot|^{\alpha-n}*|\cdot|^{\beta-n}(x)=\frac{c_{n-\alpha-\beta}c_\alpha c_\beta}{c_{\alpha+\beta}c_{n-\alpha}c_{n-\beta}}|x|^{\alpha+\beta-n}\quad \text{with }c_\alpha:=\frac{\Gamma(\frac \alpha 2)}{\pi^{\alpha/2}},$$
which can be found for instance in \cite[page 134]{LL}.

To prove \eqref{greenrapr} consider first $u\in C^\infty_c(\Omega)$. Writing
$$(-\Delta)^\frac{s}{2} u=(-\Delta)^\frac{\sigma}{2}\circ \underbrace{(-\Delta)\circ \dots\circ(-\Delta)}_{k\textrm{ times}} u,$$
and using \eqref{eqgreenb} $k+1$ times one obtains
\[
\begin{split}
u(x)&=\int_\Omega G_2(x,y_1)(-\Delta)u(y_1)dy_1\\
&=\dots\\
&=\int_\Omega G_2(x,y_1)\int_\Omega G_2(y_1,y_2)\dots\int_\Omega G_\sigma(y_k,y) (-\Delta)^\frac{s}{2} u(y) dy\,dy_k\dots dy_1\\
&=\int_\Omega G_s(x,y)(-\Delta)^\frac{s}{2} u(y)dy,
\end{split}
\]
where in the last identity we used Fubini's theorem.

When $u$ is not smooth one can proceed by approximation, again using Fubini's theorem, exactly as in Proposition \ref{propgreen}.
\end{proof}

\begin{rmk}\label{navier} The Green function used in \eqref{greenrapr} is with respect to the Navier-type boundary condition
$$(-\Delta)^j u=0\text{ on } \de\Omega \text{ for }0\le j\le k-1,\quad (-\Delta)^ku\equiv 0\text{ in }\R^{n}\setminus \Omega. $$
\end{rmk}



\noindent \emph{Proof of Theorems \ref{MT2} and \ref{MT3}.} Let $u\in \tilde H^{\frac{n}{p},p}(\Omega)$.
Setting $f:=|(-\Delta)^\frac{n}{2p} u|\big|_{\Omega}\in L^p(\Omega)$ and using Proposition \ref{propgreen2}, we bound
\begin{equation}\label{urapr}
|u(x)| \le  \int_\Omega |G_\frac{n}{p}(x,y)|f(y)dy \le K_{n,\frac np}I_\frac{n}{p}f(x),
\end{equation}
where $I_{\frac np}$ is defined as in Theorem \ref{trmadams}.
Then, assuming that $\|f\|_{L^p(\Omega)}\le 1$ we can apply Theorem \ref{trmadams} and get
$$\int_\Omega e^{\alpha_{n,p}|u|^{p'}}dx\le \int_\Omega e^{\frac{n}{\omega_{n-1}}|I_\frac{n}{p}f|^{p'}}dx\le c_{n,p}|\Omega|.$$
Theorem \ref{MT3} follows analogously applying Theorem \ref{XZ}. One only needs to notice that every $u\in \tilde H^{\frac np,(p,q)}(\Omega)$ can be approximated (by definition, see paragraph before \eqref{defHspq}) by functions $u_k\in C^\infty_c(\Omega)$ which satisfy $u_k\to u$ and $(-\Delta)^\frac{n}{2p} u_k\to (-\Delta)^\frac{n}{2p}u$ in $L^{(p,q)}(\Omega)$, hence also in $L^1(\Omega)$, since $L^{(p,q)}(\Omega)$ embeds continuously into $L^1(\Omega)$ when $\Omega$ has finite measure and $p>1$. Therefore Proposition \ref{propgreen2} can be applied. 
For instance for the case $q=\infty$, still using \eqref{urapr}, and assuming that $\|f\|_{L^{(p,\infty)}}\le 1$, we bound for $\beta<\beta_{n,p,\infty}$
\[\begin{split}
\int_\Omega e^{\beta|u|}dx&\le\int_\Omega e^{\frac{\beta}{\beta_{n,p,\infty}}\gamma_{n,p,\infty} |I_\frac np f|}dx\\
&\le \frac{d_{n,p}|\Omega|}{\gamma_{n,p,\infty}-\frac{\beta}{\beta_{n,p,\infty}}\gamma_{n,p,\infty}}\\
&=\frac{K_{n,\frac{n}{p}}^{-1}d_{n,p}|\Omega|}{\beta_{n,p,\infty}-\beta}.
\end{split}\]
The cases $q=1$ and $q\in (1,\infty)$ are very similar.

The sharpness of the constants $\beta_{n,p,q}$ for $q\in (1,\infty]$ (this includes $\alpha_{n,p}=\beta_{n,p,p}$) follows from Proposition \ref{sharp} below. Indeed, up to a translation and rescaling we can assume that $B_1\subset\Omega$. Then $C^\infty_c(B_1)\subset \tilde H^{\frac np,(p,q)}(\Omega)$, and Proposition \ref{sharp} gives the desired conclusion.
\hfill $\square$


\medskip

\begin{prop}\label{sharp}
Let $p\in (1,\infty)$, $q\in (1,\infty]$. Then 
\begin{equation}\label{supinf}
 \sup_{u\in C^\infty_c(B_1): \|(-\Delta)^\frac{n}{2p}u\|_{L^{(p,q)}(\R^n)}\le 1} \int_{B_1} e^{\beta |u|^{q'}} dx= \infty
\end{equation}
for any $\beta>\beta_{n,p,q}$, where $\beta_{n,p,q}=(\alpha_{n,p})^\frac{q'}{p'}$ is as in Theorem \ref{MT3}.
\end{prop}
\begin{proof} 
We argue by contradiction. Fix $\beta > \beta_{n,p,q}$ and assume that the supremum in \eqref{supinf} is finite.
For some $\rho\in (0,\frac18]$ to be fixed later and for a cut-off function $\theta\in C^\infty_c(B_1)$ with $\theta\equiv 1$ in $B_\frac12$, consider an arbitrary function $f \in C^\infty_c(B_\rho)$ with $\vrac{f}_{L^{(p,q)}(B_\rho)} \leq 1$ and set
$$\tilde{u} :=K_{n,\frac np} \theta I_{\frac np} f,\quad u := \frac{\tilde{u}}{1+\eps},$$
for an $\eps$ such that
$$\tilde \beta := \frac{\beta}{(1+\eps)^{q'}} > \beta_{n,p,q}.$$
With the help of Lemma \ref{green2} we now write
\[\begin{split}
(-\Delta)^\frac{n}{2p} \tilde{u} &= (-\Delta)^\frac{n}{2p}(F_\frac{n}{p}* f- K_{n,\frac np}(1-\theta)I_{\frac np} f)\\
&= f - K_{n,\frac np}(-\Delta)^\frac{n}{2p} ((1-\theta)I_{\frac np} f),
\end{split}
\]
and with Proposition \ref{propdisj} below we bound
\[
 \|K_{n,\frac np} (-\Delta)^\frac{n}{2p} ((1-\theta)I_{\frac np} f)\|_{L^{(p,q)}(\R^n)} \leq \bar C \rho^{\frac n{p'}} \vrac{f}_{L^{(p,q)}(B_\rho)}.
\]
Choose now $\rho\in (0,\frac14]$ such that
$$\frac{1+\bar C\rho^{\frac{n}{p'}}}{1+\eps}\le 1.$$
Then with the triangle inequality we get
$$\|(-\Delta)^\frac{n}{2p} u\|_{L^{(p,q)}(\R^n)} \leq 1.$$
Thus,
\[
 \int_{B_\rho} e^{\tilde \beta K_{n,\frac np}^{q'}|I_{\frac np}f|^{q'}}dx \leq \int_{B_1} e^{\tilde \beta K_{n,\frac np}^{q'} |\theta I_{\frac np}f|^{q'}} dx =\int_{B_1} e^{\beta |u|^{q'}}dx,
 \]
 and since $\tilde \beta K_{n,\frac np}^{q'}> \left(\frac{n}{\omega_{n-1}}\right)^\frac{q'}{p'}$ this is a contradiction to Theorem \ref{trmadams} when $q=p$ and Theorem \ref{XZ} in general.
\end{proof}

\begin{prop}\label{propdisj} Let $\theta \in C^\infty_c(B_1)$ with $\theta\equiv 1$ in $B_\frac12$. Given $t\ge 0$, $s\in (0,n)$, $\rho\in (0,\tfrac18]$, $p\in (1,\infty)$ and $q\in (1,\infty]$, we have 
\begin{equation}\label{stimadisj}
\|(-\Delta)^\frac t2 ((1-\theta) I_s f)\|_{L^{(p,q)}(\R^n)} \le C \rho^\frac{n}{p'}\|f\|_{L^{(p,q)}(B_\rho)},\quad \text{for every } f\in C^\infty_c(B_\rho),
\end{equation}
for a constant $C$ not depending on $f$.
\end{prop}

\begin{proof} 
By duality it suffices to prove 
\begin{equation}\label{stimadisj}
\|I_s((1-\theta)(-\Delta)^\frac{t}{2} g)\|_{L^{(p',q')}(B_\rho)}\le C \rho^\frac{n}{p'}\|g\|_{L^{(p',q')}(\R^n)},\quad \text{for every }g\in \mathcal{S}.
\end{equation}
This in turn follows from the estimate
\begin{equation}\label{stimadisj2}
\|I_s((1-\theta)(-\Delta)^\frac{t}{2} g)\|_{L^\infty(B_\frac 18)}\le C \|g\|_{L^{(p',q')}(\R^n)}, \quad \text{for every }g\in \mathcal{S},
\end{equation}
since by H\"older's inequality in Lorentz spaces (see \cite{one}) we have
\[\begin{split}
\|I_s((1-\theta)(-\Delta)^\frac{t}{2} g)\|_{L^{(p',q')}(B_\rho)}& \le \|1\|_{L^{(p',q')}(B_\rho)} \|I_s((1-\theta)(-\Delta)^\frac{t}{2} g)\|_{L^\infty(B_\frac 18)}\\ 
&\le C \rho^\frac{n}{p'}  \|I_s((1-\theta)(-\Delta)^\frac{t}{2} g)\|_{L^\infty(B_\frac 18)}.\end{split}\]
Estimate \eqref{stimadisj2} in turn follows as in Lemma 3.6 of \cite{MMS}, with minor modifications, as we shall now see.
Set $\theta_1:=(1-\theta)$, fix $\theta_2\in C^\infty_c(B_\frac14)$ with $\theta_2\equiv 1$ in $B_\frac18$. In particular the supports of $\theta_1$ and $\theta_2$ are disjoint with distance at least $\frac 18$.
Consider
$$k(x,y):=\frac{\theta_1(y)\theta_2(x)}{|x-y|^{n-s}},$$
which is smooth thanks to the disjointness of the supports of $\theta_1$ and $\theta_2$.
Then we can write for $x\in B_\frac18$ and $g\in\mathcal{S}(\R^n)$
\[\begin{split}
I_s((1-\theta)(-\Delta)^\frac{t}{2} g)(x)& = \theta_2(x)I_s(\theta_1(-\Delta)^\frac{t}{2} g)(x)\\
&=\int_{\R^n} k(x,y) (-\Delta)^\frac{t}{2}g(y)dy\\
&= \int_{\R^n} (-\Delta_y)^\frac{t}{2}k(x,y) g(y)dy\\
&= \int_{\R^n} \tilde k(x,y) g(y)dy,
\end{split}
\]
where $\tilde k(x,y):=(-\Delta_y)^\frac{t}{2}k(x,y)$ is smooth and decays like $|y|^{-n-t}$ as $y\to\infty$ uniformly with respect to $x\in B_\frac18$ (see e.g. Lemma 3.5 in \cite{MMS}).
In particular from H\"older's inequality (see \cite[Theorem 3.4]{one}) we get
\[\begin{split}
\sup_{x\in B_\frac18} \int_{\R^n} \tilde k(x,y) g(y)dy&\le \sup_{x\in B_\frac18} \|\tilde k(x,\cdot)\|_{L^{(p,q)}(\R^n)} \|g\|_{L^{(p',q')}(\R^n)}\\
&\le C \|g\|_{L^{(p',q')}(\R^n)}
\end{split}\]
and \eqref{stimadisj2} follows.
\end{proof}

\section{Some applications}\label{sec:3}

A consequence of Theorem \ref{MT2} is the existence of conformal metrics on $\Omega\subset\R^n$ with prescribed $Q$-curvature $K$ i.e. solutions to  
the equation
\begin{equation}\label{eqK}
(-\Delta)^\frac{n}{2} u=Ke^{nu}\quad \text{in }\Omega,
\end{equation}
particularly in the case $n$ odd.

Here the set $\Omega$ is required to have finite measure. For existence results for \eqref{eqK} in the case $\Omega=\R^n$, $n$ odd, we refer to \cite{hyd} and \cite{JMMX}, when $n$ is even we refer to \cite{CC}, \cite{CK}, \cite{H-M}, \cite{mar0}, \cite{W-Y}.

\begin{trm}\label{trmexK} Let $\Omega\subset\R^n$ be open and have finite measure. For any  $K\in L^p(\Omega)$, $p>1$, there is a solution $u\in \tilde H^{\frac n2,2}(\Omega)+\R$ to \eqref{eqK}.
\end{trm}

\begin{proof} We only sketch the proof, since the details are similar to those in the proof of Theorem 3.1 in \cite{MS}. Consider the set
$$A:=\left\{u\in \tilde H^{\frac n2,2}(\Omega):\int_\Omega K (e^{nu}-1)dx>0\right\}\ne \emptyset,$$
and the functional
$$F(u)= \Lambda \|u\|^2 -\log \left(\int_\Omega K(e^{nu}-1)dx \right),\quad \|u\|^2:=\int_{\R^n}|(-\Delta)^\frac{n}{4} u|^2 dx,$$
where $\Lambda$ is any constant such that 
$\Lambda>\frac{n^2 p'}{4\alpha_{n,2}}$, and $\alpha_{n,2}$ is as in Theorem \ref{MT2}. Then, bounding
$$np'u \le \frac{\alpha_{n,2}u^2}{\|u\|^2}+ \frac{(\|u\| np')^2}{4\alpha_{n,2}},$$
using H\"older's inequality and Theorem \ref{MT2} with $p=2$ we get
\[
\begin{split}
\log \left(\int_\Omega K(e^{nu}-1)dx \right)&\le \log(\|K\|_{L^p(\Omega)})+\log(\|e^{nu}-1\|_{L^{p'}(\Omega)})\\
&\le C+\frac{1}{p'}\log\left(\int_\Omega e^{nup'}dx\right)\\
& \le C  +\frac{\log(c_{n,p}|\Omega|)}{p'}+\frac{n^2p'}{4\alpha_{n,2}}\|u\|^2,
\end{split}
\]
so that
$$F(u)\ge \left(\Lambda-\frac{n^2 p'}{4\alpha_{n,2}}\right)\|u\|^2-\tilde C,$$
i.e. $F$ is well-defined, bounded from below and coercive on $\tilde H^{\frac n2,2}(\Omega)$. A minimizing sequence $u_k$ is therefore bounded in $\tilde H^{\frac{n}{2},2}(\Omega)$, hence weakly converging to a minimizer $u_0\in A$. By taking first variations (the set $A$ is open) it follows that
$$(-\Delta)^\frac{n}{2} u_0= \frac{nKe^{nu_0}}{2\Lambda\int_\Omega K(e^{nu_0}-1)dx},$$
and up to adding a constant we find a solution to \eqref{eqK}.
\end{proof}

One can also prove existence results for more general semilinear equations, say
\begin{equation}\label{eqf}
(-\Delta)^\frac{n}{2} u=f(u)\quad \text{in }\Omega
\end{equation}
with $u\in \tilde H^{\frac{n}{2},2}(\Omega)$ and $f$ critical or subcritical, in the spirit for instance of the works of Adimurthi \cite{adi} and Iannizzotto and Squassina \cite{IS}, even in the case of the fractional $p$-Laplacian, but we will not do that. We only remark that in the case when $f$ is critical, e.g. $f(u)= ue^{u^2}$, a crucial ingredient is \eqref{stimaMT2c} for $p=a=2$, which is known only in dimension $1$ (by \cite{IMM}) and in even dimension (by \cite{ada}), hence the critical case in odd dimension $\ge 3$ is open. In arbitrary even dimension we mention the work of Lakkis \cite{lak}. The subcritical case should instead present no major difficulties since the functional corresponding to \eqref{eqf} should satisfy the Palais-Smale condition. For the regularity theory of nonlinear nonlocal equations we refer the reader e.g. to \cite{KMS1}, \cite{KMS2}, \cite{RS} and \cite{Sil}.

\section{Open questions}\label{sec:4}

An interesting question in whether fractional Adams-Moser-Trudinger inequalities hold for some domains of infinite measure, in the spirit of the results of Gianni Mancini-Sandeep \cite{MS} and Battaglia-Gabriele Mancini \cite{BMa} who in dimension $2$ and in the classical case $p=2$ proved that the inequality
\begin{equation}\label{stimaBG}
\sup_{u\in C^\infty_c(\Omega): \|\nabla u\|_{L^2}\le 1} \int_\Omega\left(e^{4\pi u^2}-1\right)dx<\infty
\end{equation}
holds if and only if $\lambda_1(\Omega)>0$, where $\lambda_1(\Omega)$ is the first eigenvalue of the Laplace operator with Dirichlet boundary conditions on $\Omega$ (open set in $\R^2$).

\medskip

Another natural question is whether one can replace the spaces $\tilde H^{\frac{n}{p},p}(\Omega)$ with the spaces $\tilde W^{\frac{n}{p},p}(\Omega)$, defined via a double integral (see e.g. \eqref{Wspace} below). This appears to be unknown already in dimension $1$, except when $p=2$ (see Proposition \ref{HW}).

\medskip

As already discussed in the introduction it would be interesting to prove the sharpness of the constants in the stronger form
$$\sup_{u\in \tilde H^{\frac{n}{p},p}(\Omega), \;\|(-\Delta)^{\frac{n}{2p}}u\|_{L^{p}(\Omega)}\leq 1}\int_{\Omega} f(|u|)e^{\alpha_{n,p} |u|^{p'}}dx =\infty
$$
for any function
$$f:[0,\infty)\to [0,\infty) \quad \text{with }\lim_{t\to\infty}f(t)=\infty,$$
as already known in the non-fractional case and in dimension $1$.



\appendix

\section{Appendix}

\subsection{Some useful results}

The following density result is due to Yu. V. Netrusov \cite{net}, see \cite[Thm. 10.1.1]{AH} .

\begin{trm}[Netrusov]\label{lemmadens} For $s>0$ and $p\in [1,\infty)$ the sets $C^\infty_c(\Omega)$ ($\Omega\subset\R^n$ open set) is dense in $\tilde H^{s,p}(\Omega)$.
\end{trm}

The following way of computing the fractional Laplacian of a sufficiently regular function will be used. For a proof see e.g. \cite[Prop. 2.4]{Sil}.

\begin{prop}\label{lapint} For an open set $\Omega\subset\R^n$, let $\sigma\in (0,1)$ and $u\in L_\sigma(\R^n)\cap C^{0,\alpha}(\Omega)$ for some $\alpha\in  (\sigma,1]$, or $\sigma\in [1,2)$ and $u\in L_\sigma(\R^n)\cap C^{1,\alpha}(\Omega)$ for some $\alpha\in  (\sigma-1,1]$ . Then $((-\Delta)^\frac{\sigma}{2} u)|_\Omega\in C^0(\Omega)$ and
$$(-\Delta)^\frac{\sigma}{2} u(x)=C_{n,\sigma} P.V.\int_{\R^n} \frac{u(x)-u(y)}{|x-y|^{n+\sigma}}dy:= C_{n,\sigma}\lim_{\ve\to 0}\int_{\R^n\setminus B_\ve(x)}\frac{u(x)-u(y)}{|x-y|^{n+\sigma}}dy$$
for every $x\in \Omega$. This means that 
$$\langle (-\Delta)^\frac{\sigma}{2} u, \varphi\rangle =C_{n,\sigma}\int_{\R^n} \varphi(x)\,P.V.\int_{\R^n} \frac{u(x)-u(y)}{|x-y|^{n+\sigma}}dy\,dx,\quad \text{for every } \varphi\in C^\infty_c(\Omega).$$
\end{prop}

\subsection{Hilbert space techniques}

\begin{prop}\label{HW}
For $\sigma\in (0,2)$ we have $[u]_{W^{\frac \sigma 2,2}(\R^n)}<\infty$ if and only if $(-\Delta)^\frac{\sigma}{4}u\in L^2(\R^n)$, and in this case
\begin{equation}\label{Wspace}
[u]_{W^{\frac \sigma 2,2}(\R^n)}:=\left(\int_{\R^n}\int_{\R^n}\frac{(u(x)-u(y))^2}{|x-y|^{n+\sigma}}dxdy\right)^\frac12= C_{n,\sigma} \|(-\Delta)^\frac{\sigma}{4}u\|_{L^2(\R^n)}.
\end{equation}
In particular
$$H^{\frac \sigma 2,2}(\R^n)= W^{\frac \sigma 2,2}(\R^n):=\left\{u\in L^2(\R^n):[u]_{W^{\frac \sigma 2,2}(\R^n)}<\infty\right\}.$$
\end{prop}

\begin{proof}
See e.g. Proposition 4.4 in \cite{DNPV}.
\end{proof}

\medskip

Define the bilinear form
$$\M{B}_\sigma(u,v)=\int_{\R^n}\int_{\R^n} \frac{(u(x)-u(y))(v(x)-v(y))}{|x-y|^{n+\sigma}}dxdy,\quad \text{for }u,v\in H^{\frac\sigma2,2}(\R^n),$$
where the double integral is well defined thanks to H\"older's inequality and Proposition \ref{HW}. 

The following simple and well-known existence result proves useful.

\begin{prop}\label{trmexist}
Let $\Omega\subset\R^n$ be open and have finite measure. Given $\sigma\in (0,2)$, $f\in L^2(\Omega)$ and $g:\R^n\to \R$ such that
\begin{equation}\label{intg}
\int_\Omega\int_{\R^n}\frac{(g(x)-g(y))^2}{|x-y|^{n+\sigma}}dxdy<\infty,
\end{equation}
there exists a unique function $u\in \tilde H^{\frac\sigma2,2}(\Omega)+g$ solving the problem
\begin{equation}\label{uvar}
\M{B}_\sigma(u,v)=\int_{\R^n} f v dx\quad \text{for every }v\in \tilde H^{\frac\sigma2,2}(\Omega).
\end{equation}
Moreover such $u$ satisfies $(-\Delta)^\frac{\sigma}{2} u=\frac{C_{n,\sigma}}{2}f$ in $\Omega$ in the sense of distributions, i.e.
\begin{equation}\label{uweak}
\int_{\R^n} u(-\Delta)^\frac{\sigma}{2} \varphi dx=\frac{C_{n,\sigma}}{2}\int_{\R^n} f\varphi dx\quad \text{for every }\varphi \in C^\infty_c(\Omega),
\end{equation}
where $C_{n,\sigma}$ is the constant in Proposition \ref{lapint}.

Conversely if $u\in \tilde H^{\frac \sigma2,2}(\Omega)+g$ satisfies \eqref{uweak}, then it also satisfies \eqref{uvar}.
\end{prop}

\begin{proof} The first part follows by the abstract Dirichlet principle, see e.g. \cite[Theorem 3.2]{GM}. Indeed it is easy to verify that $\|v\|_{L^2(\Omega)}\le C[v]_{W^{\frac\sigma2,2}(\R^n)}$ for $v\in \tilde W^{\frac\sigma2,2}(\Omega)$, so that 
$$\|v\|_{H}^2:=\M{B}_\sigma(v,v)=[v]_{W^{\frac\sigma2,2}(\R^n)}^2$$
is an equivalent norm on the Hilbert space $H:=\tilde W^{\frac\sigma2,2}(\Omega)$. Also notice that the linear functional
$$L:H\to\R,\quad 
L(v):=\int_{\Omega}fvdx-\M{B}_\sigma(g,v)$$
is bounded, since by H\"older's inequality, the symmetry of $\mathcal{B}_\sigma$, and the vanishing of $v$ outside $\Omega$ we bound
\[
\begin{split}
|L(v)|&\le \|f\|_{L^2(\Omega)}\|v\|_{L^2(\Omega)} + 2\left(\int_\Omega\int_{\R^n}\frac{(g(x)-g(y))^2}{|x-y|^{n+\sigma}}dxdy\right)^\frac{1}{2}[v]_{W^{\frac\sigma2,2}(\R^n)}\\
&\le C(f,g)\|v\|_{H}.
\end{split}
\]
Then, by the Dirichlet principle the functional
$$\M{F}(v):=\frac{1}{2}\|v\|_H^2-L(v)$$
has a minimizer $\bar v$, and it follows at once that $u:=\bar v+g$ solves \eqref{uvar}.
To show that $u$ also solves \eqref{uweak} we notice that $C^\infty_c(\Omega)\subset \tilde H^{\frac\sigma2,2}(\Omega)$ and with Proposition \ref{lapint} we get
\[\begin{split}
\int_{\R^n} u(-\Delta)^\frac\sigma2 \varphi dx &=C_{n,\sigma}\int_{\R^n} u(x) P.V.\int_{\R^n}\frac{\varphi(x)-\varphi(y)}{|x-y|^{n+\sigma}}dydx\\
&=\frac{C_{n,\sigma}}{2}\M{B}_\sigma(u,\varphi)\\
&= \frac{C_{n,\sigma}}{2}\int_{\R^n} f\varphi dx,
\end{split}\]
where in the second identity we used the symmetry of $|x-y|^{n+\sigma}$.

The same computation shows that \eqref{uweak} implies
$$\M{B}_\sigma(u,\varphi)=\int_{\R^n} f\varphi dx\quad \text{for every }\varphi\in C^\infty_c(\Omega),$$
which in turn implies \eqref{uvar} thanks to the density result of Netrusov, Theorem \ref{lemmadens}.
\end{proof}


The following maximum principle is a special case of Theorem 4.1 in \cite{FKV}. We recall its proof because in our case it is very simple. 

\begin{prop}\label{maxprinc2} Let $\Omega\subset\R^n$ be open and have finite measure. Let $\sigma\in (0,2)$ and $u\in \tilde H^{\frac\sigma2,2}(\Omega)+g$ solve \eqref{uvar} for some $f\in L^2(\Omega)$ with $f\ge 0$ and $g$ satisfying \eqref{intg} and $g\ge 0$ in $\Omega^c$. Then $u\ge 0$.
\end{prop}

\begin{proof} From Proposition \ref{HW} it easily follows $v:=\min\{u,0\}\in \tilde H^{\frac\sigma2,2}(\Omega)$. Then, setting $u^+:=\max\{u,0\}$, according to \eqref{uvar} we have
\[\begin{split}
0\ge \M{B}_\sigma(u,v)&=\int_{\R^n{}}\int_{\R^n{}} \frac{(u^+(x)+v(x)-u^+(y)-v(y))(v(x)-v(y))}{|x-y|^{n+\sigma}}dxdy\\
&\ge\int_{\R^n}\int_{\R^n} \frac{(v(x)-v(y))^2}{|x-y|^{n+\sigma}}dxdy,
\end{split}\]
where we used that $u^+(x)v(x)=0$, $u^+(y)v(x)\le 0$ and $u^+(x)v(y)\le 0$ for $x,y\in\R^n$. It follows at once that $v\equiv 0$, hence $u\ge 0$.
\end{proof}

\end{document}